\documentclass[10pt]{amsart}
\usepackage{amssymb}
\usepackage{amsfonts}
\usepackage{amscd}
\usepackage[dvips]{graphicx}
\usepackage[all,cmtip]{xy}
\usepackage{hyperref}
\usepackage{yhmath}
\usepackage{bbm}
\usepackage{ulem}
\usepackage{mathdots}
\usepackage{leftidx}
\usepackage{mathrsfs}
\usepackage{amsmath,amssymb, amsthm}
\usepackage{color}
\usepackage{tikz}
\usepackage{picture}
\usepackage{enumerate}
\usepackage{makecell}
\usepackage{extarrows}
\usepackage{graphicx}
\usepackage{caption}
\usepackage{subcaption}
\usepackage{float}
\usepackage{epstopdf}
\usepackage{mathtools}
\usepackage{textcase}
\usepackage{wrapfig}
\usepackage{url}
\usepackage{citeref}
\numberwithin{equation}{section}

\DeclareFontFamily{U}{cal}{}
\DeclareFontShape{U}{cal}{m}{n}{<->cmsy10}{}

\DeclareSymbolFont{rcal}{U}{cal}{m}{n}
\DeclareSymbolFontAlphabet{\mathcal}{rcal}

\newtheorem{Them}{Theorem}[section]
\newtheorem{Lem}[Them]{Lemma}
\newtheorem{Def}[Them]{Definition}
\newtheorem{Cor}[Them]{Corollary}
\newtheorem{Prop}[Them]{Proposition}
\newtheorem{Ex}[Them]{Example}
\newtheorem{Rem}[Them]{Remark}

\newcommand{\add}{{\mathsf{add}}}
\newcommand{\Add}{{\mathsf{Add}}}
\newcommand{\Prod}{{\mathsf{Prod}}}

\newcommand{\copd}{{\mathsf{coprod}}}
\newcommand{\Copd}{{\mathsf{Coprod}}}
\newcommand{\End}{{\mathsf{End}}}

\newcommand{\ind}{{\mathsf{ind}}}

\newcommand{\rad}{{\mathsf{rad}}}
\newcommand{\soc}{{\mathsf{soc}}}

\newcommand{\D}{{\mathsf{D}}}

\newcommand{\m}{\mathsf{mod}}
\newcommand{\M}{\mathsf{Mod}}
\newcommand{\stmod}{ \mathsf{\underline{mod}}}
\newcommand{\Stmod}{ \mathsf{\underline{Mod}}}
\newcommand{\Hom}{{\mathsf {Hom}}}
\newcommand{\Coker}{{\mathsf {Coker}}}
\newcommand{\Ker}{{\mathsf {Ker}}}

\newcommand{\StHom}{\mathsf{\underline{Hom}}}

\title[Simple-minded systems and weakly simpleminded systems]{A note on simple-minded systems and weakly simple-minded systems  over self-injective algebras}

\author{Zhen Zhang}

\address{Zhen Zhang
\newline Faculty of Arts and Sciences 
\newline Beijing Normal University 
\newline  Zhuhai 519087
\newline P.R.China}
\email{zhangzhen@bnu.edu.cn}

\date{version of \today}

\newcommand{\sdp}{\times\kern-.2em\vrule height1.1ex depth-.05ex}

 \normalbaselines
\begin{document}
\renewcommand{\thefootnote}{\alph{footnote}}
\setcounter{footnote}{-1} 
\renewcommand{\thefootnote}{\alph{footnote}}
\setcounter{footnote}{-1}
\setcounter{footnote}{-1}

\begin{abstract}
Let $A$	be a self-injective algebra over an algebraically closed field. We study the relationship between simple-minded systems and weakly simple-minded systems in $A$-$\stmod$. We present a  necessary and sufficient  condition for an orthogonal system to be a simple-minded system over domestic Brauer graph algebras. As a byproduct, we construct a class of simple-minded systems over 2-domestic Brauer graph algebras. 
\end{abstract}

\maketitle
\section{Introduction}
Simple-minded systems, introduced by Koenig and Liu \cite{KL}, are families  of objects that satisfy orthogonality and generating conditions in the stable module category of any artin algebra.  Dugas \cite{Dugas} defined simple-minded systems in any Hom-finite Krull-Schmidt triangulated category. The authors \cite{KL,CKL,CLZ,GLYZ,Z} studied various properties of simple-minded systems, including the finiteness of their cardinality, their invariance under stable equivalences,  a characterization of simple-minded systems over a  representation-finite self-injective algebra, and the location of the members of a simple-minded system in a quasi-tube over a self-injective algebra, among others.  
	
Weakly simple-minded systems  were  introduced by Koenig and Liu \cite{KL} for  comparison with simple-minded systems. They showed that a family of objects $\mathcal{S}$ in $A$-$\stmod$ is a simple-minded system if and only if  $\mathcal{S}$ is a weakly simple-minded system over a representation-finite self-injective algebra $A$.   However,  there exists a weakly simple-minded system that is not a simple-minded system in the stable module category of a representation-infinite  self-injective algebra; see \cite{Z}.  
In this paper, we present a sufficient condition for a weakly simple-minded system to be a simple-minded system over a self-injective algebra.

\begin{Them}{\rm(Theorem  \ref{wsms and sms}}{\rm)}\label{wsms and sms-0}
Let $A$	be a  self-injective algebra and let $\mathcal{S}$ be a finite  orthogonal system such that  $\Sigma(\mathcal{S})\subseteq\mathcal{F}(\mathcal{S})$ in $A$-$\stmod$. If $\mathcal{S}$  is a weakly simple-minded system, that is, ${\mathcal{S}^{\perp}}=\{0\}$ in $A$-$\stmod$ and the condition {\rm(}\ref{small-module-large-module}{\rm)}  holds, then $\mathcal{S}$ is a simple-minded system in $A$-$\stmod$.   
\end{Them}

We present the following  necessary and sufficient condition for an orthogonal system to be a simple-minded system over domestic Brauer graph algebras, and we construct a class of simple-minded systems over 2-domestic Brauer graph algebras.
	
\begin{Them}{\rm(Theorem  \ref{BGA-sms}}{\rm)} \label{main-them-2}
Let $A$ be a domestic Brauer graph algebra and  $\mathcal{S}$ an  orthogonal system in $A$-$\stmod$.  Then $\mathcal{S}$ is a simple-minded system  in  $A$-$\stmod$ if and only if  $\mathcal{S}$ satisfies the following conditions:
\begin{enumerate}[$(1)$]
	\item $\mathcal{S}$ contains at least one non-periodic $A$-module, 		
	\item $\Omega(\mathcal{S})$ is  contained in  $\mathcal{F}(\mathcal{S})$.
\end{enumerate}
\end{Them}	

\begin{Them}{\rm(Theorem  \ref{BGA-sms-con}}{\rm)} \label{main-them-3}
Let $A$ be a 2-domestic Brauer algebra with  Brauer graph $G$ such that $G$ has a unique cycle of even length and the number of {\rm(}additional{\rm)} edges on the inside of the cycle equals number of {\rm(}additional{\rm)}  edges on the outside of the cycle.  Let $M$ be a non-zero non-periodic indecomposable $A$-module. Then the family of objects $\mathcal{S}$ constructed in Subsection 4.2 is a simple-minded system in $A$-$\stmod$.
\end{Them}

This paper is organized as follows. In Section 2, we recall simple-minded system, and related definitions and conclusions. In Section 3, we prove Theorem \ref{wsms and sms} and Theorem \ref{BGA-sms}. In Section 4, we present an explicit construction of  simple-minded systems over a 2-domestic Brauer graph algebra, see Theorem \ref{BGA-sms-con}.

\section{Preliminaries}	
All rings are assumed to be associative with identity and all algebras  are considered to be finite-dimensional  over an algebraically closed field $k$. For a ring $R$, we denote by $R$-$\m$ the category of finitely presented left $R$-modules and  $R$-$\M$ the category of all left $R$-modules. The stable module category of $R$-$\m$ (resp. $R$-$\M$) is denoted by $R$-$\stmod$ (resp.  $R$-$\Stmod$), it has the same class of objects with $R$-$\m$ (resp.  $R$-$\M$), and for two objects $X,Y$ in $R$-$\stmod$ (resp. $R$-$\Stmod$), the abelian group   $\StHom_R(X,Y)$ from $X$ to $Y$ is the quotient  $\StHom_R(X,Y)/\mathcal{P}(X,Y)$, where $\mathcal{P}(X,Y)$ is the subgroup of $\StHom_R(X,Y)$ consisting of all $R$-module homomorphisms  which factor through a projective  $R$-module.  We denote by $\Omega$ (resp. $\Sigma$) the syzygy functor (resp. cosyzygy functor) which assigns to any object $M$ of $R$-$\stmod$ the kernel of its projective cover $P_{R}(M)\twoheadrightarrow M$ (resp. cokernel of its injective envelope $M\hookrightarrow I_{R}(M)$) in $R$-$\m$. 
	
Let $\mathcal{T}$ be an additive category with  small products and coproducts. Let  $\mathcal{A}$ be a full subcategory of $\mathcal{T}$. $\copd(\mathcal{A})$ is the smallest full subcategory $\mathcal{S}\subseteq\mathcal{T}$ containing $\mathcal{A}$ which is closed under finite direct sums and  extensions.  $\Copd(\mathcal{A})$ is the smallest full subcategory $\mathcal{S}\subseteq\mathcal{T}$ containing $\mathcal{A}$ which is closed under  small coproducts and  extensions. $\Prod(M)$ is all direct summands of arbitrary direct sums of modules from $M$.  Let  $M$ be an object of $\mathcal{T}$.  We  denote by $\add(M)$ (resp. $\Add(M)$) all direct summands of finite (resp. arbitrary) direct sums of modules from $M$. Note that, for a class of objects $\mathcal{M}=\{M_{i}\mid i\in I\}$ in  $\mathcal{T}$, we simply denote $\add(\bigoplus_{i\in I}M_{i})$ by $\add(\mathcal{M})$.

\subsection{Simple-minded systems}
Let $\mathcal{T}$ be a  triangulated category with the suspension functor $\Sigma $.   For any families $\mathcal{S}_{1}, \mathcal{S}_{2}$ of objects in $\mathcal{T}$, we define a family of objects
\[\mathcal{S}_{1}\star\mathcal{S}_{2}:=\{ X\in \mathcal{T}\mid \mbox{ There is a  triangle }S_{1} \longrightarrow  X \longrightarrow S_{2} \longrightarrow \Sigma S_{1}, where \ S_{1}\in \mathcal{S}_{1}, S_{2}\in \mathcal{S}_{2}\}. \]
	
Using the octahedral axiom,  the operator $\star$ is associative, that is,  $(\mathcal{S}_{1}\star\mathcal{S}_{2})\star\mathcal{S}_{3}=\mathcal{S}_{1}\star(\mathcal{S}_{2}\star\mathcal{S}_{3})$ for $\mathcal{S}_{1}, \mathcal{S}_{2}$ and $\mathcal{S}_{3}\subseteq \mathcal{T}$.  We denote $(\mathcal{S})_{0}=\{0\}$, and for $n\in\mathbb{Z}^{+}$, we inductively define $(\mathcal{S})_{n}=(\mathcal{S})_{n-1}\star(\mathcal{S}\cup\{0\})$. We have $(\mathcal{S})_{n}\star(\mathcal{S})_{m}=(\mathcal{S})_{n+m}$ for any non-negative integers $m$ and $n$. Similarly, one can define $ _{n}(\mathcal{S})$, and we have $(\mathcal{S})_{n}$=$_{n}(\mathcal{S})$.  We say that $\mathcal{S}$ is {\it extension-closed}, if $\mathcal{S}\star\mathcal{S}\subseteq \mathcal{S}$. One denotes the {\bf extension closure} of a family $\mathcal{S}$ of objects in $\mathcal{T}$ as $$\mathcal{F}(\mathcal{S}):=\bigcup_{n\geq0}(\mathcal{S})_{n},$$ which is the smallest extension closed full subcategory of  $\mathcal{T}$  containing $\mathcal{S}$.
	
For any family $\mathcal{S}$ of objects in $\mathcal{T}$, we set 
$$\mathcal{S^{\perp}}:=\{Y\in \mathcal{T} \mid\mathcal{T}(X,Y)=0, \forall X\in \mathcal{S}\},$$
$$\mathcal{^{\perp}S}:=\{Y\in \mathcal{T}\mid\mathcal{T}(Y,X)=0, \forall X\in \mathcal{S}\}.$$ 
We know that both $\mathcal{S^{\perp}}$ and $\mathcal{^{\perp}S}$ are extension closed subcategories of $\mathcal{T}$  as well as closed under direct summands. We shall denote $\mathcal{S^{\perp}}\cap\mathcal{^{\perp}S}$ by $\mathcal{^{\perp}S^{\perp}}$, which is called a {\it stable bi-perpendicular category}. Note that $(\mathcal{^{\perp}S})^{\perp}$ (resp. $^{\perp}(\mathcal{S}^{\perp})$) has a different meaning with stable bi-perpendicular category $\mathcal{^{\perp}S^{\perp}}$, $\mathcal{S}$ is contained in $(\mathcal{^{\perp}S})^{\perp}$ (resp. $^{\perp}(\mathcal{S}^{\perp})$), but not contained in $\mathcal{^{\perp}S^{\perp}}$.

\begin{Def}\label{brick-orthogonal-system}
Let $\mathcal{T}$ be an additive $k$-category.  An object $M$ in $\mathcal{T}$ is a {\bf stable brick} if $\mathcal{T}(M,M)\cong k$.   Moreover, a family $\mathcal{S}$ of stable bricks in $\mathcal{T}$  is an {\bf orthogonal system} if $\mathcal{T}(M,N)=0$ for all distinct  $M, N$ in $\mathcal{S}$.
\end{Def}
	
\begin{Rem}\label{F-S-properties}
If $\mathcal{S}$  is an orthogonal system, then $\mathcal{F}(\mathcal{S})$ is  closed under direct summands. The reader may refer to \cite[Lemma 2.7]{Dugas} for more details.
\end{Rem}

\begin{Def}\label{torsion-pair} 
Let $\mathcal{T}$ be a  triangulated category and let $\mathcal{X}, \mathcal{Y}$ be two additive subcategories of $\mathcal{T}$ which are closed under direct summands and isomorphisms. The pair $(\mathcal{X},\mathcal{Y})$ is called a {\bf torsion pair}, if 
\begin{enumerate}[$(1)$]
	\item $\mathcal{T}(\mathcal{X},\mathcal{Y})=0$.
	\item For any $C\in\mathcal{T}$, there exists a  triangle		
			
	\[\xymatrix{ X\ar[r]& C\ar[r] &  Y\ar[r]& \Sigma(X)}\]
with $X\in\mathcal{X}$ and $Y\in\mathcal{Y}$.
\end{enumerate}
\end{Def}
	
\begin{Rem}\label{homogeneous-tube-tir}
The torsion pair in Definition \ref{torsion-pair} is different from the torsion pair in \cite[Ch. II, Sec. 3, Def.  3.1]{BR}, the condition $\Sigma(\mathcal{X})\subseteq\mathcal{X}$ is not included in our definition. 
\end{Rem}

\begin{Def}\label{definition-sms-right-tir} {\rm(\cite[Definition 2.1]{KL}, \cite[Definition 2.4, 2.5]{Dugas})} 
Let $\mathcal{T}$ be a  triangulated category.  A family of objects $\mathcal{S}$ in $\mathcal{T}$ is a {\bf simple-minded system} if the following two conditions are satisfied$\colon$
\begin{enumerate}[$(1)$]
	\item {\rm(Orthogonality)} $\mathcal{S}$ is an orthogonal system in $\mathcal{T}$. 
	\item {\rm(Generating condition)} Extension closure $\mathcal{F}(\mathcal{S})$ of $\mathcal{S}$ is equal to $\mathcal{T}$.
\end{enumerate}
\end{Def}

\begin{Def}\label{definition-wsms-right-tir} {\rm(\cite[Definition 5.3]{KL})} 
Let $\mathcal{T}$ be a triangulated category.  A family of objects $\mathcal{S}$ in $\mathcal{T}$ is a {\bf weakly simple-minded system}  if the following two conditions are satisfied$\colon$
\begin{enumerate}[$(1)$]
	\item {\rm(Orthogonality)} $\mathcal{S}$ is an orthogonal system in $\mathcal{T}$. 
	\item {\rm(Weakly generating condition)} For any non-zero object $X\in\mathcal{T}$, there is an object $S$ in $\mathcal{S}$ such that $\mathcal{T}(S,X)\ncong 0.$
\end{enumerate}
\end{Def}
	
\begin{Rem}\label{wsms-sms}
\begin{enumerate}[$(1)$]
	\item It follows from Koenig-Liu \cite{KL}  that simple-minded systems coincide with weakly simple-minded systems in the stable module category of any representation-finite self-injective algebra. 
	\item We \cite{Z} present an example that there is a weakly simple-minded system  which is not a simple-minded system in the stable module category of a $2$-domestic Brauer graph algebra. 
\end{enumerate}
\end{Rem}
	
We state two necessary conditions of simple-minded systems as follows.
\begin{Them}$($\cite[Theorem 3.3]{Dugas}$)$\label{subset-of-sms}
Let $\mathcal{T}$ be a Hom-finite Krull-Schmidt triangulated category. Suppose $\mathcal{X}\subseteq\mathcal{S}$ for a simple-minded system $\mathcal{S}$ in $\mathcal{T}$. Then $(^{\perp}\mathcal{X},\mathcal{F}(\mathcal{X}))$ and $(\mathcal{F}(\mathcal{X}),\mathcal{X}^{\perp})$ are torsion pairs in $\mathcal{T}$.  In particular, $ \mathcal{F}(\mathcal{X})$ is a functorially finite subcategory of $\mathcal{T}$.
\end{Them}	
	
\begin{Them}$($\cite[Theorem 1.3]{CLZ}$)$\label{simple-module-and-n-tube}
Let $A$ be a self-injective algebra and $\mathcal{C}$ a quasi-tube of rank $n$. Then the number of elements in a simple-minded system of $A$ lying in $\mathcal{C}$ is strictly less than $n$. In particular, none of the indecomposable modules in a simple-minded system  lie in a homogeneous tube of the AR-quiver.
\end{Them}	
The following theorem provides a practical method to verify when an orthogonal system is a simple-minded system.
\begin{Them}$($\cite[Theorem 3.1]{Z}$)$\label{sufficient-condition-of-s.m.s.}
Let $A$ be a domestic Brauer graph algebra and $\mathcal{S}$ an orthogonal system in  $A$-$\stmod$. Then $\mathcal{S}$ is a simple-minded system, if and only if $\mathcal{S}$ contains at least one object of each Euclidean component and the {\rm(}finite{\rm)} set of quasi-simples of all quasi-tubes {\rm(}except homogeneous tubes consisting of band modules{\rm)} is contained in $\mathcal{F(S)}$. 
\end{Them}

\section{ A study on a necessary condition of simple-minded systems}

\subsection{Simple-minded system and weakly simple-minded system} Properties of a simple-minded system have been studied extensively, including the finite cardinality of a simple-minded system,  invariance of a simple-minded system under stable equivalences and  a sufficient and necessary condition of a simple-minded system over representation-finite self-injective  algebras, and so on. In this subsection, we shall focus on a necessary condition as follows. If $\mathcal{S}$ is a simple-minded system, then 
$\Omega(\mathcal{S})$ is  contained in $\mathcal{F}(\mathcal{S})$. Furthermore,   $\Omega({\mathcal{F}(\mathcal{S})})$ is still contained  in $\mathcal{F}(\mathcal{S})$ (refer to Lemma \ref{Omega-F(S)}), it is a necessary condition of a simple-minded system. We shall study when this necessary condition could be sufficient.

\begin{Lem}\label{Omega-F(S)}
Let $A$ be a self-injective algebra and $\mathcal{S}$  an  orthogonal system  in $A$-$\stmod$.  Then $\Omega(\mathcal{S})$ is contained in  $\mathcal{F}(\mathcal{S})$ if and only if $\Omega(\mathcal{F}(\mathcal{S}))$ is contained in  $\mathcal{F}(\mathcal{S})$.
\end{Lem}
\begin{proof}
It is clear that necessity holds. To prove sufficiency, it suffices to show that each  $\Omega(\mathcal{(S)}_{n})$ is contained in $\mathcal{F}(\mathcal{S})$ for every positive integer $n$.  We induct  on positive integer $n$. It is easy to see that $\Omega(\mathcal{(S)}_{0})=\{0\}$ and $\Omega(\mathcal{(S)}_{1})=\Omega(\mathcal{S})$ are contained in $\mathcal{F}(\mathcal{S})$. Without loss of generality, we assume our conclusion holds for $n=k-1$, we prove that  $\Omega(\mathcal{(S)}_{k})$ is contained in $\mathcal{F}(\mathcal{S})$. By \cite[Lemma 2.6]{Dugas},  we know that for an object $X\in \mathcal{(S)}_{k}$, there is a triangle as follows.
	\begin{equation}\label{triangle-k}
		X_{1}\xrightarrow{}X\xrightarrow{}S_{1}\xrightarrow{} X_{1}[1], 
	\end{equation} 
	where $S_{1}\in\mathcal{S}$  and  $X_{1}\in\mathcal{(S)}_{k-1}$.
	
	Applying  functor $\Omega$ to the triangle (\ref{triangle-k}), we get the following triangle:
	\begin{equation}\label{triangle-k-1}
		\Omega(X_{1})\xrightarrow{}\Omega(X)\xrightarrow{}\Omega(S_{1})\xrightarrow{}  \Omega(X_{1})[1].
	\end{equation} 
	
	Since  both $\Omega(\mathcal{S})$ and $\Omega(\mathcal{(S)}_{k-1})$ are contained  in $\mathcal{F}(\mathcal{S})$, by induction, 	both $\Omega(S_{1})$ and $\Omega(X_{1})$ are  in $\mathcal{F}(\mathcal{S})$. Since $\mathcal{F}(\mathcal{S})$ is closed under extensions, $\Omega(X)$ is in $\mathcal{F}(\mathcal{S})$. It follows that $\Omega(\mathcal{(S)}_{k})$ is contained in $\mathcal{F}(\mathcal{S})$.  Hence our induction holds,  thus sufficient condition holds.
\end{proof}

Dually, we have 

\begin{Lem}
Let $A$ be a self-injective algebra and  $\mathcal{S}$  an  orthogonal system  in $A$-$\stmod$.  Then $\Sigma(\mathcal{S})$ is contained in  $\mathcal{F}(\mathcal{S})$ if and only if $\Sigma(\mathcal{F}(\mathcal{S}))$ is contained  in  $\mathcal{F}(\mathcal{S})$.
\end{Lem}
Recall that, for a family of object $\mathcal{S}$, $\mathcal{S}$ is said to be {\bf  Nakayama-stable} provided that $\nu(\mathcal{S})=\mathcal{S}$, where $\nu=\D(A)\otimes_{A}-$ is Nakayama functor and $\D=\Hom_{k}(-,k)$ is $k$-duality.
\begin{Lem} \label{omega-self-injective-1}
Let $A$ be a self-injective algebra and let $\mathcal{S}=\{S_{i}\mid i\in I\}$ be a Nakayama-stable orthogonal system in $A$-$\stmod$.  Then $\Omega(\mathcal{S})$ is contained in $\mathcal{F}(\mathcal{S})$ if and only if $\Sigma(\mathcal{S})$ is contained in  $\mathcal{F}(\mathcal{S})$. 	
\end{Lem}	
\begin{proof}
We prove only that if $\Omega(\mathcal{S})$ is contained in $\mathcal{F}(\mathcal{S})$, then $\Sigma(\mathcal{S})$ is also contained in  $\mathcal{F}(\mathcal{S})$. The converse part is dual. Take  $S_{j}$ in $\mathcal{S}$ for $j\in I$.  Since $\Omega(\mathcal{S})$ is contained in $\mathcal{F}(\mathcal{S})$, there is  a positive integer $m$  such that $\Omega(S_{j})$ is in  $ (\mathcal{S})_{m}$ for some integer $m$. Since $\Omega(S_{j})\in (\mathcal{S})_{m}$,  there is a non-split triangle as follows.
	
\begin{equation}\label{seq-1}
S_{i}\xrightarrow{\alpha}\Omega(S_{j})\xrightarrow{}M\xrightarrow{} S_{i}[1],
\end{equation}
\noindent where $S_{i}\in\mathcal{S}$ and $M\in (\mathcal{S})_{m-1}$.
	
We claim that $S_{i}=\nu^{-1}(S_{j})$. Otherwise,
 \[\StHom_A(S_{i}, \Omega(S_{j}))\cong\D\StHom_A(\Omega(S_{j}), \nu\Omega(S_{i}))\cong\StHom_A(S_{j}, \nu(S_{i}))\\\cong 0.\] 
 Thus triangle (\ref{seq-1}) splits, it is a contradiction. 
Rotating the above triangle (\ref{seq-1}) to the right twice,  we have 
\[M\xrightarrow{\alpha}\nu^{-1}(S_{j})[1]\xrightarrow{}S_{j}\xrightarrow{} M[1].\] 	
Since  $M\in (\mathcal{S})_{m-1}$ and $S_{j}\in\mathcal{S}$, $\Sigma(\nu^{-1}(S_{j}))=\nu^{-1}(S_{j})[1]$ is in $(\mathcal{S})_{m}$. 
Since Nakayama functor $\nu$  is an equivalence and it is  also a permutation on $\mathcal{S}$,  it is clear that $\{\Sigma(\nu^{-1}(S_{j}))\mid j\in I\}=\Sigma(\mathcal{S})$. Hence 
$\Sigma(\mathcal{S})$ is  contained in  $\mathcal{F}(\mathcal{S})$.
\end{proof}

\begin{Cor}\label{two-out-of-three}
Let $A$ be a self-injective algebra and  $\mathcal{S}$  an  orthogonal system  in $A$-$\stmod$.   Then $\Omega(\mathcal{S})$ {\rm(}resp. $\Sigma(\mathcal{S})${\rm)} is contained in  $\mathcal{F}(\mathcal{S})$, if  and only if  $\mathcal{F}(\mathcal{S})$ satisfies two-out-of-three condition, that is, for the objects $X$, $Y$ and $Z$ in a triangle
\[X\xrightarrow{}Y\xrightarrow{}Z\xrightarrow{}  X[1],\]
if any two terms  are  in $\mathcal{F}(\mathcal{S})$,   so is the third term. 
\end{Cor}

\begin{Cor} \label{omega-self-injective-2}
Let $A$ be a self-injective algebra and let $\mathcal{S}=\{S_{i}\mid i\in I\}$ be a Nakayama-stable orthogonal system in $A$-$\stmod$.  If $\Omega(\mathcal{S})$ is contained  in $\mathcal{F}(\mathcal{S})$, then $\mathcal{F}(\mathcal{S})$ is a triangulated subcategory of $A$-$\stmod$. 	
\end{Cor}	
	
Let $\mathcal{T}$ be a triangulated category with suspension functor $\Sigma$. Recall that  a pair of strictly full subcategories $(\mathcal{S},\mathcal{S}')$ is a {\bf t-structure} on $\mathcal{T}$  provided that $(\mathcal{S},\mathcal{S}')$ is a torsion pair and $\Sigma\mathcal{S}\subseteq\mathcal{S}$,  $\Sigma^{-1}\mathcal{S}'\subseteq\mathcal{S}'$.
Let $A$ be a self-injective algebra. 
Let $\mathcal{S}$ be a full  subcategory of  $A$-$\Stmod$. We recall some notations as follows.
\begin{enumerate}[(1)]	
	\item $\copd(\mathcal{S})$ is the smallest full subcategory $\mathcal{R}\subseteq A$-$\Stmod$, which is closed under finite direct sums and satisfying \[\mathcal{S}\subseteq\mathcal{R},\ \  \  \  \  \  \   \mathcal{R}\star\mathcal{R}\subseteq\mathcal{R}.\]
		
	\item $\Copd(\mathcal{S})$ is the smallest full subcategory $\mathcal{R}\subseteq A$-$\Stmod$, which is closed under  small coproducts and  satisfying \[\mathcal{S}\subseteq\mathcal{R},\ \  \  \  \  \  \  \mathcal{R}\star\mathcal{R}\subseteq\mathcal{R}.\]
\end{enumerate}
	
\begin{Lem}\label{copod-A}
Let $A$ be a self-injective algebra and $\mathcal{S}$ an orthogonal system in $A$-$\stmod$.
Then  $\Copd(\mathcal{S})=\Copd(\mathcal{F}(\mathcal{S}))$$($resp. $\copd(\mathcal{S})=\copd(\mathcal{F}(\mathcal{S}))$$)$.	
		
\end{Lem}
\begin{proof}
We prove only $\Copd(\mathcal{S})=\Copd(\mathcal{F}(\mathcal{S}))$.  First, it is clear that  $\Copd(\mathcal{S})\subseteq\Copd(\mathcal{F}(\mathcal{S}))$.
We prove the converse part. Since  $\mathcal{S}\subseteq\Copd(\mathcal{S})$ and  $\Copd(\mathcal{S})\star\Copd(\mathcal{S})\subseteq\Copd(\mathcal{S})$, $(\mathcal{S})_{2}$ is contained in $\Copd(\mathcal{S})$. By induction, we have $\mathcal{F}(\mathcal{S})\subseteq\Copd(\mathcal{S})$. It follows from definition that $\Copd(\mathcal{S})$ is closed under extensions and small coproducts.  Since $\Copd(\mathcal{F}(\mathcal{S}))$ is  the smallest full subcategory of $A$-$\Stmod$ containing $\mathcal{F}(\mathcal{S})$ which is closed under extensions and coproducts, $\Copd(\mathcal{F}(\mathcal{S}))\subseteq\Copd(\mathcal{S})$. 
\end{proof}
	
Let $A$ be a self-injective algebra and $\mathcal{S}$  an orthogonal system in $A$-$\stmod$. We assume the following  condition are true:
\begin{equation}\label{small-module-large-module}
\text{If}\   {\mathcal{S}^{\perp}}=\{0\}\  \text{in}\ A\text{-}\stmod,\  \text{then}\  {\mathcal{S}^{\perp}}=\{0\}\  \text{in}\  A\text{-}\Stmod.  
\end{equation}
	
\begin{Them}\label{wsms and sms}
Let $A$	be a  self-injective algebra and let $\mathcal{S}$ be a finite  orthogonal system such that  $\Sigma(\mathcal{S})\subseteq\mathcal{F}(\mathcal{S})$ in $A$-$\stmod$. If $\mathcal{S}$  is a weakly simple-minded system, that is, ${\mathcal{S}^{\perp}}=\{0\}$ in A-$\stmod$ and  the above condition {\rm(}\ref{small-module-large-module}{\rm)} holds, then $\mathcal{S}$ is a simple-minded system in $A$-$\stmod$.   
\end{Them}
\begin{proof}
We assume that $\mathcal{S}$ is a finite  orthogonal system, $\Sigma(\mathcal{S})\subseteq\mathcal{F}(\mathcal{S})$ and ${\mathcal{S}^{\perp}}=\{0\}$ in A-$\stmod$. If $\Sigma(\mathcal{S})\subseteq\mathcal{F}(\mathcal{S})$, then $\Sigma(\mathcal{F}(\mathcal{S}))\subseteq\mathcal{F}(\mathcal{S})$. Thus, by \cite[Lemma 1.2.2]{CNS}, $\Sigma(\Copd(\mathcal{F}(\mathcal{S})))\subseteq \Copd(\mathcal{F}(\mathcal{S})).$ It follows from Lemma \ref{copod-A} that  $\Sigma(\Copd(\mathcal{S}))\subseteq \Copd(\mathcal{S}).$ By  \cite[Theorem 2.3.3]{CNS}, Canonaco, Neeman and Stellari proved that the pair  $(\Copd(\mathcal{S}),\Copd(\mathcal{S})^{\perp})$ is a t-structure in {\tiny } $A$-$\Stmod$.  Note that ${\Copd(\mathcal{S})^{\perp}}={\mathcal{S}^{\perp}}$.
Since ${\mathcal{S}^{\perp}}=\{0\}$ in A-$\stmod$, by condition (\ref{small-module-large-module}), ${\mathcal{S}^{\perp}}=\{0\}$ in A-$\Stmod$,  thus $\Copd(\mathcal{S})=A$-$\Stmod$.  
		
We claim that  $\mathcal{F}(\mathcal{S})=\Copd(\mathcal{F}(\mathcal{S}))\cap A$-$\stmod$. Indeed, it is clear that 	$\Copd(\mathcal{F}(\mathcal{S}))\cap A$-$\stmod=\copd(\mathcal{F}(\mathcal{S}))=\copd({\mathcal{S}})$.
Note that $\mathcal{F}(\mathcal{S})$ is the smallest extensions closed full  subcategory of $A$-$\stmod$ containing 
$ \mathcal{S}$. Note that $\mathcal{F}(\mathcal{S})$ is closed under finite  coproducts.  It follows from definition that  $\mathcal{F}(\mathcal{S})=\copd({\mathcal{S}}).$
Thus  $\mathcal{F}(\mathcal{S})=\Copd(\mathcal{F}(\mathcal{S}))\cap A$-$\stmod=\Copd(\mathcal{S})\cap A$-$\stmod=A$-$\Stmod\cap A$-$\stmod=A$-$\stmod$. Thus $\mathcal{S}$ is a simple-minded system in  $A$-$\stmod$. 
\end{proof}
	
\begin{Cor}\label{wsms and sms-1}
Let $A$	be a representation-finite self-injective algebra and  $\mathcal{S}$  an  orthogonal system in A-$\stmod$. If $\mathcal{S}$  is a weakly simple-minded system, then $\mathcal{S}$ is a simple-minded system in $A$-$\stmod$.   
\end{Cor}
\begin{proof}
 It is known that, if $A$ is of finite type, then any indecomposable module is finite dimensional and any module is a direct sum of indecomposable modules.  Thus condition (\ref{small-module-large-module}) holds for representation-finite self-injective algebras and our conclusion follows from Theorem \ref{wsms and sms}.
\end{proof}

If $A$ is a representation-finite self-injective algebra, we proved \cite[Theorem 3.1]{GLYZ} that a family of objects $\mathcal{S}$ is a simple-minded system in $A$-$\stmod$ if and only if  $\mathcal{S}$ is an orthogonal system, $\mathcal{S}$  is Nakayama-stable (that is, $\nu(\mathcal{S})=\mathcal{S}$), and the cardinality of $\mathcal{S}$ is equal to the number of non-isomorphic simple $A$-modules. In the following proposition, we shall replace cardinality condition with the condition that 
$\Omega(\mathcal{S})$ is contained in  $\mathcal{F}(\mathcal{S})$.

\begin{Prop} \label{sms-omega}
Let $A$ be a representation-finite self-injective algebra and  $\mathcal{S}$ a family of objects in $A$-$\stmod$. Then $\mathcal{S}$  is a simple-minded system  in  $A$-$\stmod$ if and only if 
\begin{enumerate}[$(1)$]
	\item $\mathcal{S}$ is an  orthogonal system  in  $A$-$\stmod$.
	\item $\mathcal{S}$  is Nakayama-stable. 
	\item  $\Omega(\mathcal{S})$ is contained in  $\mathcal{F}(\mathcal{S})$.
\end{enumerate}
\end{Prop}
	
\begin{proof}
It  follows from \cite[Theorem 3.1]{GLYZ} that necessary condition holds. We show the sufficient one. Let $S$ be an object of $\mathcal{S}$. Consider an almost split sequence as follows.
\[\tau S\xrightarrow{}\oplus_{i=1}^{m}X_{i}\xrightarrow{}S\xrightarrow{}  \nu\Omega(S).\]
By  Lemma \ref{two-out-of-three}, we know that $\oplus_{i=1}^{m}X_{i}$ is in  $\mathcal{F}(\mathcal{S})$. Since $\mathcal{F}(\mathcal{S})$ is closed under direct summands, each $X_{i}$ is in $\mathcal{F}(\mathcal{S})$. Thus each object which has an irreducible map to $S$ is in $\mathcal{F}(\mathcal{S})$.
		
Since $\Omega(\mathcal{S})$ is contained  in  $\mathcal{F}(\mathcal{S})$, by Lemma \ref{omega-self-injective-1}, $\Sigma(\mathcal{S})$ is contained in  $\mathcal{F}(\mathcal{S})$. Since $\Sigma(\mathcal{S})$ is contained in  $\mathcal{F}(\mathcal{S})$, it can be  proved similarly that each object that receives an irreducible map from $\mathcal{S}$ is in $\mathcal{F}(\mathcal{S})$.  Therefore each object in any given connected subquiver of AR-quiver $\Gamma_{A}$ is in $\mathcal{F}(\mathcal{S})$. It is well-known that $A$ is of finite representation type if and only if AR-quiver $\Gamma_{A}$ of $A$ is a finite connected quiver. Therefore $A$-$\ind\subset \mathcal{F}(\mathcal{S})$ and $\mathcal{F}(\mathcal{S})=A$-$\m$. Hence $\mathcal{S}$ is a simple-minded system in $A$-$\stmod$.
\end{proof}
\subsection{A  characterization of simple-minded systems over a domestic Brauer graph algebra} In this subsection, we provide a  characterization of simple-mind systems over a domestic Brauer graph algebra.
Instead of giving a detailed introduction to domestic Brauer graph algebra, we  shall  simply recall some related well-known results.  Recall that a domestic Brauer graph algebra $A$ is 1-domestic or 2-domestic. The stable AR-quiver  $_{s}\Gamma_{A}$ of $A$ consists of 
$m$ Euclidean components of the form $\mathbb{Z}\widetilde{A}_{p,q}$,
$m$ quasi-tubes of the form $\mathbb{Z}A_{\infty}/\langle\tau^{p}\rangle$,
$m$ quasi-tubes of the form $\mathbb{Z}A_{\infty}/\langle\tau^{q}\rangle$,
as well as  infinitely many homogeneous tubes of the form $\mathbb{Z}A_{\infty}/\langle\tau\rangle$, where $m=1,2$. Please refer to \cite{BoS,D1,S}  for more details about domestic Brauer graph algebras.

We recall some concepts as follows.
Let $A$ be an artin algebra and $\mathcal{C}$  a quasi-tube in the AR-quiver $\Gamma_{A}$. An indecomposable $A$-module $X$ is called a {\bf  quasi-simple} of $\mathcal{C}$, provided that it is not projective and lies at the end of $\mathcal{C}$.  We say a {\it sectional path} is a path $\cdots \to M_i \to M_{i+1} \to M_{i+2} \to \cdots$ in the AR-quiver satisfying $M_i\ncong \tau(M_{i+2})$. For a quasi-simple $X$ and any natural number $r\geq1$, up to isomorphism, there is the unique infinite sectional path  in $\mathcal{C}$ starting at $X$: 
\[
X = X(1)\rightarrow X(2)\rightarrow \cdots \rightarrow X(r)\rightarrow X(r + 1)\rightarrow \cdots,
\]
dually, up to isomorphism, there is the unique infinite sectional path in $\mathcal{C}$ ending at $X$:
\[\cdots \rightarrow [r+1]X \rightarrow [r]X \rightarrow \cdots \rightarrow [2]X \rightarrow [1]X=X.
\]
An $A$-module $X$ is said to be {\bf non-periodic} if there is no positive integer $n$ such that $\Omega^{n}(X)\cong X$. 

\begin{Def}{\rm(\cite[Chapter X, Definition 1.3]{SS})}\label{ray-point} Let $A$ be an algebra and $\mathcal{C}$  a component of the AR-quiver $\Gamma_{A}$ of $A$.  A {\bf ray point}	 of $\mathcal{C}$ is defined  to be a point in $\mathcal{C}$ such that there exists an infinite sectional path in $\mathcal{C}$
\[X=X(1)\xrightarrow{} X(2) \xrightarrow{}X(3)\xrightarrow{} \cdots \xrightarrow{}X(m)\xrightarrow{} \cdots \]
starting at $X$	 and containing all sectional paths starting at $X$. The corresponding  $A$-module $X$ is called a {\bf ray module}. The {\bf coray point} and {\bf coray module} are defined analogously.
\end{Def}
	
Note that a quasi-simple in a quasi-tube is a coray, but a coray is not necessarily a quasi-simple.   Let $A$ be a self-injective algebra and take an almost split sequence $0\xrightarrow{}X\xrightarrow{}Y\xrightarrow{f}\tau^{-1}(X)\xrightarrow{} 0$. Recall that $\alpha(\tau^{-1}(X))$ denotes the number $n$ of direct summands in a decomposition of $Y=\oplus_{i=1}^{n}Y_{i}$ into indecomposable modules.

\begin{Prop}{\rm(\cite[Corollary 3.5]{K})} \label{quasi-simple-1}Let $A$ be a self-injective artin algebra, and let  $0\xrightarrow{}X\xrightarrow{}Y\xrightarrow{f}Z\xrightarrow{} 0$ be a short exact sequence in $A$-$\m$,  which is not almost split. Suppose that one of the modules  $Y$ and $Z$ is indecomposable and $f$ is irreducible. Then the following conditions are equivalent for a natural number $n>1:$
\begin{enumerate}[$(1)$]
	\item $\alpha(\tau^{-1}(X))=n.$	
	\item $X$ is simple and $\rad P/\soc P$ decomposes into $n$ summands, where $P$ denotes the projective cover of $X$.
	\item $X$ is simple and $\rad I/\soc I$ decomposes into $n$ summands, where $I$ denotes the injective envelope  of $X$.
\end{enumerate}
Moreover, if $X$ satisfies these conditions, then $Y$ is either indecomposable projective or isomorphic to the radical of some indecomposable projective module.
\end{Prop}	
	
The following proposition provides a connection between different connected  AR-components of the AR-quiver of a representation-infinite self-injective algebra.

\begin{Prop}\label{s.m.s.-maximal-universal-module}
Let $A$ be representation-infinite self-injective algebra. Given a triangle 	
\begin{align}\label{tri-1}
X\xrightarrow{}Y\xrightarrow{\underline{f}}Z\xrightarrow{} X[1]
\end{align}
in $A$-$\stmod$, where both $Y$ and $Z$ are indecomposable non-projective, and $f:Y\rightarrow Z$ is an irreducible map in $A$-$\m$. Then  $X$ {\rm(}resp. $X[1]${\rm)} is a coray point of some stable connected AR-component. 
\end{Prop}	
\begin{proof}Note that a point is a ray point if and only it is a coray point in the stable	 AR-quiver of a self-injective algebra. Let $X$ be a point in the stable connected component $_{s}\Gamma_{1}$ of the stable AR-quiver $_{s}\Gamma_{A}$. We only  prove that $X$  is a  coray point of $_{s}\Gamma_{1}$, since $[1]: A$-$\stmod\rightarrow A$-$\stmod$ is an equivalence.  
We assume that the triangle  (\ref{tri-1}) is induced by a short exact sequence as follows.
\begin{align}\label{ses-1}
	0\xrightarrow{}X\xrightarrow{}Y\oplus P\xrightarrow{(f,h)}Z\xrightarrow{} 0.
\end{align}
Where $P$ is a projective $A$-module. If the sequence   (\ref{ses-1}) is  an almost split sequence, then $Y$ is  isomorphic to $\rad(P)/\soc(P)$ and $P$ is indecomposable. Since $Y$ is indecomposable,  $\alpha(\tau^{-1}(X))=2$. Thus $X$ is a coray point in $_{s}\Gamma_{1}$. 

Now we assume that the short exact sequence  (\ref{ses-1}) is not an almost split sequence.  

\noindent{\it Case one: $P$ is isomorphic to zero.}

We shall show that $X$ is a coray point in $_{s}\Gamma_{1}$. On the contrary, we assume that $X$ is not  a  coray point of $_{s}\Gamma_{1}$.  Let $\alpha(\tau^{-1}(X))=n>1$.   By Proposition \ref{quasi-simple-1}, $Y$ is either indecomposable projective or isomorphic to the radical of some indecomposable projective module. The former case contradicts the condition that $Y$ is non-projective. For the latter case, $Y$ is isomorphic to the radical $\rad P$ of an indecomposable projective $P$. By the equivalent conditions of Proposition \ref{quasi-simple-1},  $X$ is  a simple module and $X\cong\soc P$. Therefore $Z$ is isomorphic to $\rad P/\soc P$ and it is decomposable (with $n$ direct summands). This contradicts the condition that $Z$ is indecomposable. Hence $\alpha(\tau^{-1}(X))=1$ and then $X$  is a  coray point of $_{s}\Gamma_{1}$.

\noindent{\it Case two: $P$ is not isomorphic to zero.}

Since $f$ is an irreducible map, $f$ is a monomorphism or an epimorphism.  
If $f$ is a monomorphism,  then we take a short exact sequence as follows.
\begin{align}\label{ses-3}
0\xrightarrow{}Y\xrightarrow{f}Z\xrightarrow{} \Coker(f) \xrightarrow{}0.
\end{align}
It can be proved dually that $\Coker(f)$ is a ray point in $_{s}\Gamma_{1}$. Considering the triangle induced by short exact sequence (\ref{ses-3}), we have $X[1]\cong\Coker(f)$ in $A$-$\stmod$. Thus $X$ is a coray  point in $_{s}\Gamma_{1}$.

If f is an epimorphism, then we take a short exact sequence as follows.
\begin{align}\label{ses-4}
	0\xrightarrow{}\Ker(f)\xrightarrow{}Y\xrightarrow{f}Z\xrightarrow{} 0.
\end{align}
It is similar to the Case one that $\Ker(f)$ is a coray  point in $_{s}\Gamma_{1}$. Considering the triangle induced by short exact sequence (\ref{ses-4}), we have $X\cong \Ker(f)$ in $A$-$\stmod$. Thus $X$ is a coray point in $_{s}\Gamma_{1}$.
\end{proof}

\begin{Them} \label{BGA-sms}
Let $A$ be a domestic Brauer graph algebra and  $\mathcal{S}$ an  orthogonal system in $A$-$\stmod$.  Then $\mathcal{S}$ is a simple-minded system $A$-$\stmod$  if and only if $\mathcal{S}$ satisfies the following conditions.
\begin{enumerate}[$(1)$]
\item $\mathcal{S}$ contains at least one non-periodic $A$-module;		
\item $\Omega(\mathcal{S})$ is contained  in  $\mathcal{F}(\mathcal{S})$.
\end{enumerate}
\end{Them}	
\begin{proof}
Since the domestic Brauer graph algebra $A$ is 1-domestic or 2-domestic, $\Gamma_{A}$ contains at most two Euclidean components. It follows from Theorem \ref{sufficient-condition-of-s.m.s.} that the necessary condition holds. We prove the sufficient condition.    By Lemma \ref{omega-self-injective-1}, both $\Omega(\mathcal{S})$ and $\Sigma(\mathcal{S})$ are contained  in  $\mathcal{F}(\mathcal{S})$. Then it is clear that  $\tau(\mathcal{S})=\Omega^{2}(\mathcal{S})$ and $\tau^{-1}(\mathcal{S})=\Sigma^{2}(\mathcal{S})$ are contained in $\mathcal{F}(\mathcal{S})$. Thus  $\tau^{i}(\mathcal{S})$ is contained in $\mathcal{F}(\mathcal{S})$ for every $i\in\mathbb{Z}$.  

If $A$ is 1-domestic, then there is only one Euclidean component in $\Gamma_{A}$.  Since $\mathcal{F}(\mathcal{S})$ contains $\tau^{i}(\mathcal{S})$  for every $i\in\mathbb{Z}$ and $\mathcal{F}(\mathcal{S})$  is closed under extensions, the whole Euclidean component is contained in $\mathcal{F}(\mathcal{S})$. If $A$ is 2-domestic, there are exactly two Euclidean components in $\Gamma_{A}$.  By condition (1), we may take a non-periodic module $M$ in $\mathcal{S}$. We know that $M$ and $\Omega(M)$ are not in the same Euclidean component (please see Remark \ref{Euclidean-comp}). By condition (2), $\Omega(M)$ is in $\mathcal{F}(\mathcal{S})$. Since  $\mathcal{F}(\mathcal{S})$ contains $\tau^{i}(M)$ and  $\tau^{i}(\Omega(M))$ for every $i\in\mathbb{Z}$ and  $\mathcal{F}(\mathcal{S})$ is closed under extensions, the two Euclidean components are contained in $\mathcal{F}(\mathcal{S})$. 

For each irreducible map $X\xrightarrow{\alpha} Y$ with $X$ and $Y$ non-projective and indecomposable in an Euclidean component, we extend $\alpha$ to a  triangle 	\[X\xrightarrow{\alpha}Y\xrightarrow{}Z\xrightarrow{} X[1],\] 
and by Lemma  \ref{two-out-of-three}, $Z$ is in $\mathcal{F}(\mathcal{S})$, by Proposition  \ref{s.m.s.-maximal-universal-module}, $Z$ is a quasi-simple. By \cite[Lemma 3.8]{Z},  all quasi-simples, except homogeneous tubes consisting of band modules,  are in $\mathcal{F}(\mathcal{S})$. Thus $\mathcal{F}(\mathcal{S})$  contains  all connected
AR-components which are not homogeneous tubes.  By Theorem \ref{simple-module-and-n-tube}, there is no 
simple $A$-module  contained in any homogeneous tube. Thus $\mathcal{F}(\mathcal{S})$ contains all simple $A$-modules. Hence $\mathcal{S}$  is a simple-minded system in $A$-$\stmod$.
\end{proof}

\section{An explicit construction of simple-minded systems over 2-domestic Brauer graph algebras}
As an application of Theorem \ref{BGA-sms}, we  construct a class of simple-minded systems over a 2-domestic Brauer graph algebra in this section. 
Let $A$ be a 2-domestic Brauer graph algebra with Brauer graph $G$ with $n$ edges. It follows that  the stable AR-quiver $_{s}\Gamma_{A}$ of $A$  consists of two stable Euclidean components $\Gamma_{1}$ and $\Gamma_{2}$ of the form $\mathbb{Z}\widetilde{A}_{p,q}$, two quasi-tubes of rank $p$ and  two quasi-tubes of rank $q$, as well as infinitely many homogeneous tubes.  
	
\begin{Rem}\label{Euclidean-comp}
We recall some properties of Euclidean components of the AR-quiver $\Gamma_{A}$ as follows {\rm(}please refer to \cite[Lemma 3.5 and Lemma 3.6]{Z}{\rm)}. Two Euclidean components $\Gamma_{1}$ and $\Gamma_{2}$ of AR-quiver $\Gamma_{A}$ are  stable generalized standard,  and the stable components satisfy condition $\Omega(_{s}\Gamma_{1})={_{s}\Gamma_{2}}$. Therefore, every object in $_{s}\Gamma_{1}$ or $_{s}\Gamma_{2}$ is a stable brick in $A$-$\stmod$. Note that, if there is no composition of finitely many irreducible maps from object $M$ to object $N$ in $_{s}\Gamma_{1}${\rm(}resp.  $_{s}\Gamma_{2}${\rm)}, then $\StHom_A(M,N)=0$. 
\end{Rem}

The following theorem plays a role in the construction of simple-minded systems.
\begin{Them}$($\cite[Ch. V, Sec. 2, Cor. 2.4]{ARS}$)$ \label{omega-self-injective}
Let $A$ be an artin algebra and let $C$ be an indecomposable module with $\underline{\End}_{A}(C)$ a division ring. Then the following are equivalent for a short exact sequence $\delta:0\xrightarrow{}\tau C\xrightarrow{}B\xrightarrow{} C\xrightarrow{}0.$
\begin{enumerate}[$(1)$]
		\item $\delta$ is almost split.
		\item  $\delta$ does not split.
		\item $B\ncong \tau C\oplus C$.
\end{enumerate}
\end{Them}

We shall construct a class of simple-minded systems for the case $p=q$ as follows. 
In this case, two stable Euclidean components $_{s}\Gamma_{1}$ and $_{s}\Gamma_{2}$ are of the form $\mathbb{Z}\widetilde{A}_{p,p}$.  Take a non-zero non-periodic indecomposable object $M$ in $A$-$\stmod$. Without loss of generality, we assume $M$ is in  the stable Euclidean component $_{s}\Gamma_{1}$.  By Remark \ref{Euclidean-comp}, $M$ is a stable brick in $A$-$\stmod$. By Serre duality,
\[\StHom_A(M, M)\cong\D\StHom_A(M,\nu\Omega(M))\cong\D\StHom_A(M,\Omega(M))\ncong 0.\]
Hence there exists a non-zero morphism $\alpha:M\rightarrow \Omega(M)$ in $A$-$\stmod$. Consider  the triangle extended by morphism $\alpha$: 
\begin{equation}\label{seq-2}
M\xrightarrow{\alpha}\Omega{(M)}\xrightarrow{}W\xrightarrow{} M[1].
\end{equation} 
Rotating the triangle (\ref{seq-2}) to the left twice, we have 
\begin{equation}\label{seq-3}
 M[-2]\xrightarrow{}W[-1]\xrightarrow{}M\xrightarrow{\alpha} \Omega{(M)}.
\end{equation} 
Since $M$ is a stable brick and $M[-2]=\Omega^{2}(M)\cong\tau M$, by Theorem 
\ref{omega-self-injective}, the triangle (\ref{seq-3}) is induced by an almost split sequence ending with $M$.
Clearly, $W$ contains two indecomposable non-projective direct summands, denoted  $W_{1 1}$ and $W_{12}$. Note that $W_{1 1}$ and $W_{12}$ are in stable Euclidean component $_{s}\Gamma_{2}$.
Set  $\mathcal{S}_{1}:=\{M,W_{11}, W_{12}\}$.
\smallskip
	
\noindent{\bf Claim 1.}  $\mathcal{S}_{1}$ is an orthogonal system in $A$-$\stmod$. 

Indeed, since the connected component $\Gamma_{1}$  is stable generalized standard and $\Omega(_{s}\Gamma_{1})={_{s}\Gamma_{2}}$, by Remark \ref{Euclidean-comp}, $W_{11}$ and $W_{12}$ are mutually orthogonal in $A$-$\stmod$.  By Serre duality, 
\[\StHom_A(M, W_{11})\cong\D\StHom_A(W_{11},\nu\Omega(M))\cong\D\StHom_A(W_{11},\Omega(M))\cong\D\StHom_A(W_{11}[-1],\Omega^{2}(M)),\] 
\[\StHom_A(W_{11}, M)\cong\D\StHom_A(M,\Omega(W_{11})).\]  
Since $\Gamma_{1}$ is stable generalized standard, it follows that \[\StHom_A(W_{11}[-1],\Omega^{2}(M))\cong\StHom_A(W_{11}[-1],\tau(M))\cong 0\] and \[\StHom_A(M,\Omega(W_{11}))\cong\StHom_A(M,W_{11}[-1])\cong 0.\] 
Note that $W_{11}[-1]$ is a middle term of the almost split sequence induced by triangle (\ref{seq-3}).
Thus $M$ and $W_{11}$ are mutually orthogonal in $A$-$\stmod$. Similarly, $M$ and $W_{12}$ are mutually orthogonal in $A$-$\stmod$.  Hence $\mathcal{S}_{1}$ is an orthogonal system in $A$-$\stmod$. 
	
\medskip
	
Consider  the triangle extended by $\alpha_{1}:W_{11}\rightarrow \Omega(W_{11})$:
\begin{equation}\label{seq-4}
W_{11}\xrightarrow{\alpha_{1}}\Omega{(W_{11})}\xrightarrow{}W_{2}\xrightarrow{} W_{11}[1]. 
\end{equation} 
Rotating triangle (\ref{seq-4}) to the left twice, we have 
\begin{equation}\label{seq-5}
\tau W_{11}\xrightarrow{\alpha}W_{2}[-1]\xrightarrow{}W_{11}\xrightarrow{}\tau W_{11}[1].
\end{equation} 
Since $W_{11}$ is a stable brick and $\Omega^{2}(W_{11})\cong\tau W_{11}$, by Theorem
\ref{omega-self-injective}, the triangle (\ref{seq-5}) is induced by an almost split sequence ended with $W_{11}$.
It is clear that $W_{2}$ contains two indecomposable non-projective direct summands, denoted by $W_{21}$ and $W_{22}$. 
	
\noindent{\bf Claim 2.} One of the objects $W_{21}$  and $W_{22}$ is isomorphic to $M$. 

Since triangle (\ref{seq-3}) is induced  by an almost split sequence, there is an irreducible map from $\tau M$ to $W_{11}[-1]$ in the AR-quiver $\Gamma_{A}$. It follows that there is an irreducible map from $\Omega(M)$ to $W_{11}$. Since $W_{11}$ is a stable brick and $\Omega^{2}(W_{11})\cong\tau W_{11}$, by Theorem 
\ref{omega-self-injective}, the sequence (\ref{seq-5}) is induced by an almost split sequence. Thus $\Omega(M)$ is isomorphic to an indecomposable direct summand of $W_{2}[-1]$, then $M$ is an  indecomposable non-projective direct summands of  $W_{2}$. Hence one of objects  $W_{21}$  and $W_{22}$ is isomorphic to $M$. Without loss of generality, we assume $W_{21}\cong M$. Set $\mathcal{S}_{2}:=\{W_{11}, M, W_{22}\}$.
	
\noindent{\bf Claim 3.}  $\mathcal{S}_{2}$ is an orthogonal system in $A$-$\stmod$. 

Indeed, by Claim 1, $W_{11}$ and  $M$  are orthogonal in $A$-$\stmod$. Since $M$ and $W_{22}$ are middle terms of an almost split sequence, there is no composition of finitely many irreducible maps between $M$ and $W_{22}$. Therefore $M$ and $W_{22}$ are orthogonal in $A$-$\stmod$.  By Serre duality, 
\[\StHom_A(W_{11}, W_{22})\cong\D\StHom_A(W_{22},\nu\Omega(W_{11}))\cong\D \StHom_A(W_{22},\Omega(W_{11}))\cong\D\StHom_A(W_{22}[-1],\Omega^{2}(W_{11})),\] 
\[\StHom_A(W_{22}, W_{11})\cong\D\StHom_A(W_{11},\Omega(W_{22})).\]  
Since $\Gamma_{1}$ is stable generalized standard,  \[\StHom_A(W_{22}[-1],\Omega^{2}(W_{11}))=\StHom_A(W_{22}[-1],\tau(W_{11}))\cong 0\] and \[\StHom_A(W_{11},\Omega(W_{22}))\cong\StHom_A(W_{11},W_{22}[-1])\cong 0.\] Thus $W_{11}$ and $W_{22}$ are mutually orthogonal in $A$-$\stmod$.  Hence $\mathcal{S}_{2}$ is an orthogonal system in $A$-$\stmod$. 

\noindent{\bf Claim 4. }	$\mathcal{S}_{1}$ and $\mathcal{S}_{2}$ are mutually orthogonal in $A$-$\stmod$, that is, $\mathcal{S}_{1}\cup\mathcal{S}_{2}=\{M, W_{11}, W_{12},W_{22}\}$  is an orthogonal system in $A$-$\stmod$.

It suffices to show that $W_{12}$ and $W_{22}$ are mutually orthogonal in $A$-$\stmod$. Indeed, since $\Gamma_{1}$ is stable generalized standard, it follows that 
\[\StHom_A(W_{12}, W_{22})\cong\D\StHom_A(W_{22},\Omega(W_{12}))=0,\] and \[\StHom_A(W_{22}, W_{12})\cong\D\StHom_A(W_{12},\Omega(W_{22}))=0.\]
Thus $\mathcal{S}_{1}\cup\mathcal{S}_{2}$  is an orthogonal system in $A$-$\stmod$. The following diagram provides a visual depiction for the orthogonality of the set $\mathcal{S}_{1}\cup\mathcal{S}_{2}$.

\vspace{0cm}
\[_{s}\Gamma_{1}: 
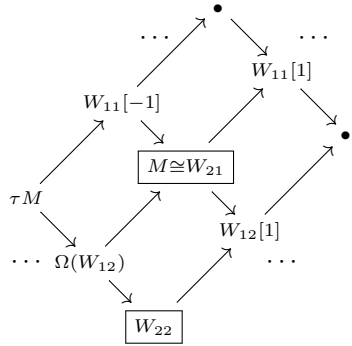
\xymatrix@dr@R=17pt@C=17pt@!0{	
	&&\cdots&& \\	
	\scriptstyle \bullet\ar[rr] &&\scriptstyle W_{11}[1]\ar[rr] &&\scriptstyle\bullet & \\
	\cdots\ \ \ \ \ \ \ &&&&&  \\
	&&&& \\
	\scriptstyle W_{11}[-1]\ar[uuu]\ar[rr]^{}&& \boxed{\scriptstyle M\cong W_{21}}\ar[uuu]\ar[rr]&&\scriptstyle W_{12}[1]\ar[uuu] &  \cdots\\
	&&&&& \\
	&&&& \\
	\scriptstyle\tau M \ar[uuu]^{}\ar[rr]_{}&  &\scriptstyle \Omega(W_{12})\ar[uuu]^{}\ar[rr]& & \boxed{\scriptstyle W_{22}}\ar[uuu] \\
&\cdots&&& \\}
{_{s}\Gamma_{2}}: 
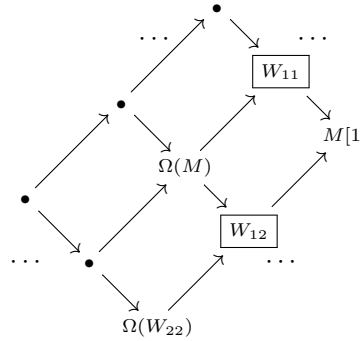
\xymatrix@dr@R=17pt@C=17pt@!0{	
&&\cdots&& \\	
	\scriptstyle \bullet\ar[rr] && \boxed{\scriptstyle W_{11}}\ar[rr] &&\scriptstyle M[1]& \\
	\cdots\ \ \ \ \ \ \ 	&&&&&  \\
	&&&& \\
	\scriptstyle \bullet\ar[uuu]\ar[rr]^{}&& \scriptstyle \Omega(M)\ar[uuu]\ar[rr]&& \boxed{\scriptstyle W_{12}}\ar[uuu] &  \cdots\\
	&&&&& \\
	&&&& \\
	\scriptstyle\bullet \ar[uuu]^{}\ar[rr]_{}&  &\scriptstyle \bullet\ar[uuu]^{}\ar[rr]& &\scriptstyle\ \Omega(W_{22})\ar[uuu] \\
&\cdots&&&	 \\}
\]

Proceeding this process, we have orthogonal systems in $A$-$\stmod$ as follows.  
\[\mathcal{S}_{1}=\{M,W_{11}, W_{12}\}, \mathcal{S}_{2}=\{W_{11}, M, W_{22}\}, \mathcal{S}_{3}=\{W_{12},M, W_{32}\},\]
\[\mathcal{S}_{4}=\{W_{22},W_{11}, W_{42}\}, \mathcal{S}_{5}=\{W_{32},W_{12}, W_{52}\}, \cdots.\] 
Since $p=q$, within a finitely many steps, $W_{(n-3)2}$ coincides with some $W_{i2}$, and  $W_{(n-2)2}$  coincides with some $W_{j2}$ for some $i$ and $j$. We get an orthogonal system in $A$-$\stmod$  as follows.
\begin{equation}\label{sms}
\mathcal{S}=\{M,W_{11},W_{12},W_{22},W_{32},\cdots,W_{(n-2)2}\}.
\end{equation}

Since two stable Euclidean components satisfy the condition $_{s}\Gamma_{1}=\Omega(_{s}\Gamma_{2})$, $\mathcal{S}$ contains at least one object for each Euclidean component. By the triangle (\ref{seq-2}), $\Omega(M)$ is in $(\mathcal{S})_{3}$, thus it is in  $\mathcal{F}(\mathcal{S})$. It follows from the construction of $\mathcal{S}$ that the set $\Omega(\mathcal{S})$ is contained  in $\mathcal{F}(\mathcal{S})$. By Theorem \ref{BGA-sms}, $\mathcal{S}$ is a simple-minded system in $A$-$\stmod$. Note that the cardinality of $\mathcal{S}$ is the number of non-projective simple $A$-modules. Summarizing the above discussion, we have conclusion as follows.

\begin{Them}\label{BGA-sms-con}
Let $A$ be a 2-domestic Brauer algebra with  Brauer graph $G$ such that  $G$ has a unique cycle of even length and the number of {\rm(}additional{\rm)} edges on the inside of the cycle equals to  number of {\rm(}additional{\rm)} edges on the outside of the cycle.  Let $M$ be a non-zero non-periodic indecomposable $A$-module. Then the family  $\mathcal{S}$ of objects  constructed as the above {\rm(}\ref{sms}{\rm)} is a simple-minded system in $A$-$\stmod$.
\end{Them}

We give an example of the  construction of simple-minded systems over a 2-domestic Brauer graph algebra shown below.
\begin{Ex}\label{BGA-construc-sms}
Consider the Brauer graph  $G=(G_{0},G_{1}, m, o)$ as follows. 
	
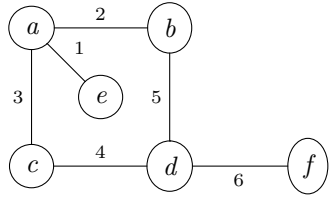
\[\xymatrix@r@R=26pt@C=26pt@!0{
		*++[o][F]\txt{a} \ar@{-}[rr]^{2}\ar@{-}[dr]^{1}& &*++[o][F]\txt{b}  & & & \\
		&*++[o][F]\txt{e} &  &    & & & \\
		*++[o][F]\txt{c} \ar@{-}[rr]^{4} \ar@{-}[uu]^{3}  & & *++[o][F]\txt{d}\ar@{-}[uu]^{5}\ar@{-}[rr]_{6} &&  *++[o][F] \txt{f}\,.  \\
}\]
We have $m(i)=1$  for each  vertex  $i\in G_{0}$.  The algebra $A\cong kQ_{G}/I_{G}$ determined by Brauer graph $G$ has quiver $ Q_{G}$ as follows {\rm(}relations omitted{\rm)}.
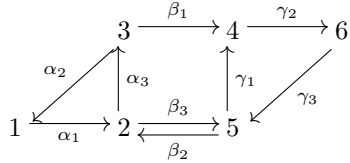
	\[\xymatrix@C=1pc@R=2pc{ 
		&& 3 \ar@<+.5ex>[dll]_{\alpha_{2}} \ar@<+.5ex>[rr]^{\beta_{1}}\save[] \restore &&4 \ar@<+.5ex>[rr]^{\gamma_{2}}&&6\ar@<+.5ex>[dll]^{\gamma_{3}}\save[]	 \restore \\
		1 \ar@<+.5ex>[rr]_{\alpha_{1}} \save[]	\restore && 2 \ar@<+.5ex>[rr]^{\beta_{3}} \ar@<+.5ex>[u]_{\alpha_{3}}  \save[]  \restore&& 5\ar@<+.5ex>[u]_{\gamma_{1}} \ar@<+.5ex>[ll]^{\beta_{2}}\save[]  \restore}\]	
		
We know that $A\cong kQ_{G}/I_{G}=~~\begin{matrix}1\\2\\3\\1\end{matrix}
\!\!\!~\oplus~~\begin{matrix}2\\\begin{matrix}3\\1\end{matrix}~~\begin{matrix}5\end{matrix}\\2\end{matrix}
\!\!\!~\oplus~~\begin{matrix}3\\\begin{matrix}1\\2\end{matrix}~~\begin{matrix}4\end{matrix}\\3\end{matrix}
\!\!\!~\oplus~~\begin{matrix}4\\\begin{matrix}3\end{matrix}~~\begin{matrix}6\\5\end{matrix}\\4\end{matrix}
\!\!\!~\oplus~~\begin{matrix}5\\\begin{matrix}2\end{matrix}~~\begin{matrix}4\\6\end{matrix}\\5\end{matrix}
\!\!\!~\oplus~~\begin{matrix}6\\5\\4\\6\end{matrix}\ .$	
Take $M=\begin{matrix}3\end{matrix}$, the simple module corresponding to vertex $3$ in $Q_{G}$. By the construction,  we obtain six orthogonal triple systems:
	
$\left\{~~\begin{matrix}3\end{matrix},~
\begin{matrix}1\\2\end{matrix},~
\begin{matrix}4\end{matrix}\right\}$,
\ 
$\left\{~~\begin{matrix}1\\2\end{matrix},~
\begin{matrix}3\end{matrix},~
\begin{matrix}2\\\begin{matrix}3\\1\end{matrix}~~\begin{matrix}5\end{matrix}\end{matrix}\right\}$,
\
$\left\{~~\begin{matrix}4\end{matrix},~
\begin{matrix}3\end{matrix},~
\begin{matrix}6\\5\end{matrix}\right\}$,
	\
$\left\{\begin{matrix}2\\\begin{matrix}3\\1\end{matrix}~~\begin{matrix}5\end{matrix}\end{matrix},~
\begin{matrix}1\\2\end{matrix},~
\begin{matrix}5\\\begin{matrix}2\end{matrix}~~\begin{matrix}4\\6\end{matrix}\end{matrix}\right\}$,
\
$\left\{\begin{matrix}6\\5\end{matrix},~
\begin{matrix}5\\\begin{matrix}2\end{matrix}~~\begin{matrix}4\\6\end{matrix}\end{matrix},~
\begin{matrix}4\end{matrix}\right\}$,
	\
$\left\{\begin{matrix}5\\\begin{matrix}2\end{matrix}~~\begin{matrix}4\\6\end{matrix}\end{matrix},~
\begin{matrix}2\\\begin{matrix}3\\1\end{matrix}~~\begin{matrix}5\end{matrix}\end{matrix},~
\begin{matrix}6\\5\end{matrix}
\right\}$.
Then the orthogonal system $\mathcal{S}=\left\{\begin{matrix}3\end{matrix},~\begin{matrix}1\\2\end{matrix},~\begin{matrix}4\end{matrix},~\begin{matrix}5\\\begin{matrix}2\end{matrix}~~\begin{matrix}4\\6\end{matrix}\end{matrix},~\begin{matrix}2\\\begin{matrix}3\\1\end{matrix}~~\begin{matrix}5\end{matrix}\end{matrix},~\begin{matrix}6\\5\end{matrix} \right\}$ is a simple-minded system in $A$-$\stmod$. Note that $\mathcal{S}$ satisfies the condition $\Omega(\mathcal{S})\subseteq\mathcal{F}(\mathcal{S}).$
\end{Ex}


At the end of this paper, we  give an  example satisfying the construction of Theorem  \ref{BGA-sms-con} over a 1-domestic Brauer graph algebra.
\begin{Ex}\label{BGA-construc-sms-1}	
Consider the Brauer graph  $G=(G_{0},G_{1}, m, o)$ shown below. 
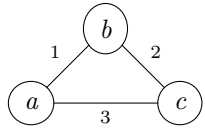
\[\xymatrix@r@R=28pt@C=28pt@!0{&*++[o][F]\txt{b} &  &    & \\
*++[o][F]\txt{a}\ar@{-}[ur]^{1}  \ar@{-}[rr]_{3}  & & *++[o][F]\txt{c}\ar@{-}[ul]_{2} \,.  \\
}\]
We have $m(i)=1$  for any  vertex  $i\in G_{0}$.  The algebra $A\cong kQ_{G}/I_{G}$ determined by Brauer graph $G$ has quiver $ Q_{G}$ as follows{\rm(}relation omitted{\rm)}. 
\[\xymatrix@C=1pc@R=2pc{ 
&& 1  \ar@<+.5ex>[drr]^{\alpha_{3}\ \ \ \ } \ar@<+.5ex>[rrrr]^{\beta_{1}}\save[] \restore &&&&2\ar@<+.5ex>[llll]^{\alpha_{1}}\ar@<+.5ex>[dll]^{\beta_{2}}\save[]	 \restore \\
&& && 3\ar@<+.5ex>[urr]^{\ \ \ \alpha_{2}}\ar@<+.5ex>[ull]^{\beta_{3}} \save[]  \restore}\]	
	
$A\cong kQ_{G}/I_{G}=~~\begin{matrix}1\\\begin{matrix}2\end{matrix}~~\begin{matrix}3\end{matrix}\\1\end{matrix}
\!\!\!~\oplus~~\begin{matrix}2\\\begin{matrix}1\end{matrix}~~\begin{matrix}3\end{matrix}\\2\end{matrix}
\!\!\!~\oplus~~\begin{matrix}3\\\begin{matrix}1\end{matrix}~~\begin{matrix}2\end{matrix}\\3\end{matrix}.$		
Take $M=\begin{matrix}1\end{matrix}$, the simple module corresponding to vertex $1$ in $Q_{G}$. The triangles in the construction of simple-minded systems are stated as follows.
\begin{equation}\label{seq-6}
1\xrightarrow{}\ \begin{matrix} \begin{matrix}2\end{matrix}~~\begin{matrix}3\end{matrix}\\1\end{matrix}\xrightarrow{}2 \oplus 3\xrightarrow{}\  \begin{matrix}1\\\begin{matrix}3\end{matrix}~~\begin{matrix}2\end{matrix}\end{matrix},
\end{equation}
\begin{equation}\label{seq-7}
2\xrightarrow{}\ \begin{matrix} \begin{matrix}1\end{matrix}~~\begin{matrix}3\end{matrix}\\2\end{matrix}\xrightarrow{}1 \oplus 3\xrightarrow{}\  \begin{matrix}2\\\begin{matrix}1\end{matrix}~~\begin{matrix}3\end{matrix}\end{matrix},
\end{equation} 
\begin{equation}\label{seq-8}
3\xrightarrow{}\ \begin{matrix} \begin{matrix}1\end{matrix}~~\begin{matrix}2\end{matrix}\\3\end{matrix}\xrightarrow{}1 \oplus 2\xrightarrow{}\  \begin{matrix}3\\\begin{matrix}1\end{matrix}~~\begin{matrix}2\end{matrix}\end{matrix}.	
\end{equation} 
By the construction,  we obtain three orthogonal triple systems: 
\[\{~~\begin{matrix}1\end{matrix},~\begin{matrix}2\end{matrix},~\begin{matrix}3\end{matrix}\},
\{~~\begin{matrix}2\end{matrix},~\begin{matrix}1\end{matrix},~\begin{matrix}3\end{matrix}\},
\{~~\begin{matrix}3\end{matrix},~\begin{matrix}1\end{matrix},~\begin{matrix}2\end{matrix}\},\]
Thus $\mathcal{S}=\{M, W_{11}, W_{12}\}=\{1,2,3\}$ is the set of all simple $A$-modules. It is a simple-minded system in $A$-$\stmod$.	Note that for any non-zero non-periodic indecomposable module $M'$ in $A$-$\stmod$, the set $\mathcal{S}=\{M', W'_{11}, W'_{12}\}$ is a simple-minded system in $A$-$\stmod$.
\end{Ex}	

\section{Declarations}
\subsection*{ Ethical Approval}
This declaration is “not applicable”.
\subsection*{ Funding} 
	
The research work is supported by NSFC (No.12301044).

\subsection*{Availability of data and materials}
Data availability is not applicable to this article as no new data were created or analyzed in this study.

\end{document}